\documentclass[12pt,a4paper]{article}
\usepackage{graphicx}
\usepackage{amsmath,amsthm,amsfonts} 
\usepackage{txfonts,bm} 

\DeclareMathOperator\id{id}

\DeclareMathOperator\GL{GL}
\DeclareMathOperator\PGL{PGL}
\DeclareMathOperator\PG{PG}
\DeclareMathOperator\Char{Char}
\DeclareMathOperator\spn{span}
\DeclareMathOperator\sgn{sgn}
\DeclareMathOperator\SL{SL}
\DeclareMathOperator\Sp{Sp}

\newcommand{\cB}{{\mathcal B}} 
\newcommand{\cH}{{\mathcal H}} 
\newcommand{\cM}{{\mathcal M}} 
\newcommand{\cL}{{\mathcal L}}
\newcommand{\cO}{{\mathcal O}}

\newcommand{\cQ}{{\mathcal Q}}
\newcommand{\cS}{{\mathcal S}}
\newcommand{\cR}{{\mathcal R}}

\newcommand{\bF}{{\mathbb F}}
\newcommand{\bP}{{\mathbb P}}

\newcommand{\vA}{{\bm A}}
\newcommand{\vE}{{\bm E}}

\newcommand{\vU}{{\bm U}}
\newcommand{\vV}{{\bm V}}
\newcommand{\vW}{{\bm W}}
\newcommand{\vX}{{\bm X}}
\newcommand{\vY}{{\bm Y}}

\newcommand{\va}{{\bm a}}
\newcommand{\vb}{{\bm b}}
\newcommand{\ve}{{\bm e}}
\newcommand{\vi}{{\bm i}}
\newcommand{\vj}{{\bm j}}
\newcommand{\vk}{{\bm k}}
\newcommand{\vs}{{\bm s}}
\newcommand{\vu}{{\bm u}}

\newcommand{\trenn}{{-\hspace{0pt}}}

\newcommand{\TProdm}{{\bigotimes_{k=1}^m \vV_k}} 
\newcommand{\TProdWm}{{\bigotimes_{k=1}^m \vW_k}} 
\newcommand{\Segre}{{\cS_{1,1,1}}}
\newcommand{\Segrem}{{\cS_{(m)}}}
\newcommand{\Alt}{[\cdot,\cdot]}
\newcommand{\Stab}[2]{G_{\cS_{(#1)}(#2)}} 

\newcommand{\oben}[1]{^{(#1)}} 

\newtheorem{thm}{Theorem}
\newtheorem{prop}{Proposition}
\newtheorem{cor}{Corollary}
{\theoremstyle{definition}
\newtheorem{exa}{Example}
}

\begin{document}

\title{On invariant notions of Segre varieties in binary projective spaces}

\author{Hans Havlicek \and Boris Odehnal \and Metod Saniga}
\date{}

\maketitle


\begin{abstract}
\noindent Invariant notions of a class of Segre varieties $\Segrem(2)$ of
$\PG(2^m - 1,2)$ that are direct products of $m$ copies of PG$(1,2)$, $m$ being
any positive integer, are established and studied. We first demonstrate that
there exists a hyperbolic quadric that contains $\Segrem(2)$ and is invariant
under its projective stabiliser group $\Stab{m}{2}$. By embedding PG$(2^m -
1,2)$ into $\PG(2^m - 1,4)$, a basis of the latter space is constructed that is
invariant under $\Stab{m}{2}$ as well. Such a basis can be split into two
subsets whose spans are either real or complex{\trenn}conjugate subspaces
according as $m$ is even or odd. In the latter case, these spans can, in
addition, be viewed as indicator sets of a $\Stab{m}{2}${\trenn}invariant
geometric spread of lines of $\PG(2^m - 1,2)$. This spread is also related with
a $\Stab{m}{2}${\trenn}invariant non-singular Hermitian variety.
\par
The case $m=3$ is examined in detail to illustrate the theory. Here, the lines
of the invariant spread are found to fall into four distinct orbits under
$\Stab{3}{2}$, while the points of $\PG(7,2)$ form five orbits.
\end{abstract}

\section{Introduction}\label{se:intro}

The present note is concerned with \emph{invariant notions} in the ambient
space of certain Segre varieties over fields of characteristic two and, in
particular, over the smallest Galois field $\bF_2$. The attribute
\emph{invariant} always refers to the stabiliser of the Segre in the projective
group of the ambient space.
\par
Our text is organised as follows: In Section~\ref{se:background} we collect
some background results about those Segre varieties $\Segrem(F)$ which are
products of $m$ projective lines over a field $F$. The next section presents an
\emph{invariant quadric} of a Segre $\Segrem(F)$ for $m\geq 2$ and a ground
field $F$ of characteristic two (Theorem~\ref{thm:1}). This quadric is regular,
of maximal Witt index, contains the given Segre, and its polarity is the
fundamental polarity of this Segre. The following sections deal with Segre
varieties $\Segrem(2)$ over $\bF_2$. By extending the ground field of the
ambient space from $\bF_2$ to $\bF_4$ we find an \emph{invariant basis}
(Theorem~\ref{thm:2}) and an \emph{invariant Hermitian variety}
(Section~\ref{se:hermite}). The theory splits according as $m$ is even or odd.
In the latter case there is an \emph{invariant geometric line spread}
(Corollary~\ref{cor:1}) which gives also rise to a spread of the invariant
quadric. We make use of our previous results in Section~\ref{se:m=3}, where we
describe certain line orbits and all point orbits of the stabiliser group of
the Segre $S_{(3)}(2)$ in terms of the invariant basis. This complements
\cite[p.~82]{glynn+g+m+g-06z}, where a completely different description of
these point orbits was given without proof.
\par
There is a widespread literature on closely related topics, like notions of
\emph{rank of multi{\trenn}dimensional arrays} \cite{comon+tb+dl+c-09a},
\emph{secant varieties of Segre varieties} (mainly over the complex numbers)
\cite{abo+o+p-09a}, the very particular properties of certain Segre varieties
over $\bF_2$ \cite{shaw-08a}, \cite{shaw+g-09a}, \emph{quantum codes}
\cite{glynn+g+m+g-06z}, and \emph{entanglement of quantum bits} in physics
\cite{doko+o-09a}, \cite{heyd-08a}, \cite{levay+vrana-08a}. The few sources
which are cited here contain a wealth of further references.

\section{Notation and background results}\label{se:background}

Let $F$ be a commutative field and let $\vV_1,\vV_2,\ldots,\vV_m$ be $m\geq 1$
two{\trenn}dimensional vector spaces over $F$. So $\bP(\vV_k)=\PG(1,F)$ are
projective lines over $F$ for $k\in\{1,2,\ldots,m\}$. We consider the tensor
product $\TProdm$ and the projective space $\bP\big(\TProdm\big)=\PG(2^m-1,F)$.
The non-zero decomposable tensors of $\TProdm$ determine the \emph{Segre
variety} (see \cite{bura-61}, \cite{hirschfeld+thas-91})
\begin{equation*}
    \cS_{\underbrace{\scriptstyle 1,1,\ldots,1}_{m}}(F)
     = \big\{F\va_1\otimes\va_2\otimes\cdots\otimes\va_m \mid
     \va_k\in \vV_k\setminus\{0\} \big\}
\end{equation*}
of $\bP\big(\TProdm\big)$. This Segre will also be denoted by $\Segrem(F)$.
\par
We recall some facts which are well known from the classical case
\cite[p.~143]{bura-61}, where $F$ is the field of complex numbers. They can
immediately be carried over to our more general settings. Given a basis
$(\ve_0\oben{k},\ve_1\oben{k})$ for each vector space $\vV_k$,
$k\in\{1,2,\ldots,m\}$, the tensors
\begin{equation}\label{eq:basis}
    \vE_{i_1,i_2,\ldots, i_m}
    :=\ve_{i_1}\oben{1} \otimes\ve_{i_2}\oben{2}\otimes\cdots\otimes\ve_{i_m}\oben{m}
    \mbox{~~~with~~~} (i_1,i_2,\ldots,i_m)\in I_m:=\{0,1\}^m
\end{equation}
constitute a basis of $\TProdm$. For any multi-index $\vi=(i_1,i_2,\ldots,
i_m)\in I_m$ the \emph{opposite} multi-index, in symbols $\vi'$, is
characterised by $i_k\neq i'_k$ for all $k\in\{1,2,\ldots,m\}$. In other words,
two multi{\trenn}indices are opposite if, and only if, their Hamming distance
is maximal.
\par
Let $f_k\in\GL(\vV_k)$ for $k\in\{1,2,\ldots,m\}$. Then
\begin{equation}\label{eq:kronecker}
    f_1\otimes f_2 \otimes \cdots \otimes f_m\in\GL\Big(\TProdm\Big)
\end{equation}
denotes their Kronecker (tensor) product. Each permutation $\sigma\in S_m$
gives rise to linear bijections $\vV_k\to \vV_{\sigma(k)}$ sending
$(\ve_0\oben{k},\ve_1\oben{k})$ to
$(\ve_0\oben{\sigma(k)},\ve_1\oben{\sigma(k)})$. Also, the symmetric group
$S_m$ acts on $I_m$ via $\sigma(\vi)=\sigma(i_1,i_2,\ldots,i_m)=
(i_{\sigma^{-1}(1)},i_{\sigma^{-1}(2)},\ldots,i_{\sigma^{-1}(m)})$. There is a
unique mapping
\begin{equation}\label{eq:f-sigma}
    f_\sigma \in {\GL}\Big(\TProdm\Big)
    \mbox{~~such that~~}\vE_\vi \mapsto \vE_{\sigma(\vi)}
    \mbox{~~for all~~} \vi\in I_m.
\end{equation}
Clearly, this $f_\sigma$ depends on the chosen bases. The subgroup of
${\GL}\big(\TProdm\big)$ preserving decomposable tensors is generated by all
transformations of the form (\ref{eq:kronecker}) and (\ref{eq:f-sigma}). It
induces the \emph{stabiliser} $\Stab{m}{F}$ of the Segre $\Segrem(F)$ within
the projective group ${\PGL}\big(\TProdm\big)$.

\par
Each of the vector spaces $\vV_k$ admits a symplectic (\emph{i.~e.},
non-degenerate and alternating) bilinear form\footnote{We use the same symbol
for all these forms. Note that $\Sp(\vV_k,\Alt)=\SL(\vV_k)$, since
$\dim\vV_k=2$ for all $k\in\{1,2,\ldots,m\}$. This coincidence of a symplectic
group with a special linear group underpins much of the mathematics used in
this article.} $\Alt : \vV_k\times \vV_k\to F$. Consequently, $\TProdm$ is
equipped with a bilinear form, again denoted as $\Alt$, which is given by
\begin{equation}\label{eq:[].basisfrei}
    \big[\va_1\otimes\va_2\otimes\cdots\otimes\va_m,\vb_1\otimes\vb_2\otimes\cdots\otimes\vb_m \big]:=
     \prod_{k=1}^m [\va_k,\vb_k] \mbox{~~for~~} \va_k,\vb_k\in \vV_k,
\end{equation}
and extending bilinearly. Like the forms on $\vV_k$ this bilinear form on
$\TProdm$ is unique up to a non-zero factor in $F$. In projective terms the
form $\Alt$ on $\TProdm$ (or any proportional one) determines the
\emph{fundamental polarity} of $\Segrem(F)$, \emph{i.~e.}, a polarity which
sends $\Segrem(F)$ to its dual. This polarity is orthogonal when $m$ is even
and $\Char F\neq 2$, but null otherwise: Indeed, it suffices to consider the
tensors of our basis (\ref{eq:basis}). Given $\vi,\vj\in I_m$ we have
\begin{eqnarray}\label{eq:[].basis1}
    [\vE_{\vi},\vE_{\vi'}]&=&\prod_{k=1}^m [\ve_{i_k}\oben{k},\ve_{i'_k}\oben{k}]
    = (-1)^m[\vE_{\vi'},\vE_{\vi}]\neq 0,\\
    \label{eq:[].basis2}
    [\vE_{\vi},\vE_{\vj}]&=&0 \mbox{~~~for all~~~}\vj\neq \vi'.
\end{eqnarray}
Hence $\Alt$ is symmetric when $m$ is even and $\Char F\neq 2$, and it is
alternating otherwise.
\par
Let $m$ be even and $\Char F\neq 2$. Then $Q:\TProdm\to F: \vX\mapsto
[\vX,\vX]$ is a quadratic form having Witt index $2^{m-1}$ and rank $2^m$. So,
the fundamental polarity of the Segre $\Segrem(F)$ is the polarity of a regular
quadric. The Segre coincides with this quadric precisely when $m=2$.

\section{The invariant quadric}\label{se:quadratic}

We now focus on the case when $F$ has characteristic two. Here $\Alt$ is a
symplectic bilinear form on $\TProdm$ for any $m\geq 1$, whence the fundamental
polarity of the Segre $\Segrem(F)$ is always null. Furthermore,
(\ref{eq:[].basis1}) simplifies to
\begin{equation}\label{eq:[].basis12}
    [\vE_{\vi},\vE_{\vi'}] = \prod_{k=1}^m [\ve_{0}\oben{k},\ve_{1}\oben{k}]
    = [\vE_{\vi'},\vE_{\vi}]\neq 0.
\end{equation}

\begin{prop}\label{prop:1}
Let $m\geq 2$ and $\Char F=2$. Then there is a unique quadratic form
$Q:\TProdm\to F$ satisfying the following two properties:
\renewcommand{\theenumi}{\emph{\arabic{enumi}}}
\begin{enumerate}
    \item $Q$ vanishes for all decomposable tensors.

    \item The symplectic bilinear form $\Alt:\TProdm\times\TProdm\to F$ is
        the polar form of $Q$.
\end{enumerate}
\end{prop}

\begin{proof}
(a) We denote by $I_{m,0}$ the set of all multi{\trenn}indices
$(i_1,i_2,\ldots,i_m)\in I_m$ with $i_1=0$. In terms of our basis
(\ref{eq:basis}) a quadratic form is given by
\begin{equation}\label{eq:def.Q}
    Q:\TProdm\to F : \vX\mapsto
    \sum_{\vi\,\in\, I_{m,0}}\frac{[\vE_\vi,\vX][\vE_{\vi'},\vX]}
                              {[\vE_\vi,\vE_{\vi'}]}.
\end{equation}
Given an arbitrary decomposable tensor we have
\begin{eqnarray*}
    Q(\va_1\otimes\cdots\otimes\va_m)
    &=&
    \sum_{\vi\,\in\, I_{m,0}}\frac{[\vE_\vi,\va_1\otimes\cdots\otimes\va_m]
                               [\vE_{\vi'},\va_1\otimes\cdots\otimes\va_m]}
                              {[\vE_\vi,\vE_{\vi'}]}\\
    &=&\sum_{\vi\,\in\, I_{m,0}}\frac{[\ve_0\oben{1},\va_1][\ve_1\oben{1},\va_1]
                               \cdots [\ve_0\oben{m},\va_m][\ve_1\oben{m},\va_m]}
                              {[\ve_0\oben{1},\ve_1\oben{1}]\cdots [\ve_0\oben{m},\ve_1\oben{m}]}\\
    &=&2^{m-1}\,\frac{[\ve_0\oben{1},\va_1][\ve_1\oben{1},\va_1]
                               \cdots [\ve_0\oben{m},\va_m][\ve_1\oben{m},\va_m]}
                              {[\ve_0\oben{1},\ve_1\oben{1}]\cdots [\ve_0\oben{m},\ve_1\oben{m}]}\\
    &=&0,
\end{eqnarray*}
where we used (\ref{eq:[].basis12}), $\# I_{m,0}= 2^{m-1}$, $m\geq 2$, and
$\Char F=2$. This verifies property 1.
\par
(b) Let $\vj,\vk\in I$ be arbitrary multi{\trenn}indices. Polarising $Q$ gives
\begin{eqnarray*}
    Q(\vE_{\vj}+\vE_{\vk})+Q(\vE_{\vj})+Q(\vE_{\vk}) &=&
    Q(\vE_{\vj}+\vE_{\vk}) + 0 + 0\\
    &=&
    \sum_{\vi\,\in\, I_{m,0}}\frac{[\vE_\vi,\vE_{\vj}+\vE_{\vk}][\vE_{\vi'},\vE_{\vj}+\vE_{\vk}]}
                              {[\vE_\vi,\vE_{\vi'}]}.
\end{eqnarray*}
The numerator of a summand of the above sum can only be different from zero if
$\vi\in \{\vj',\vk'\}$ and $\vi'\in \{\vj',\vk'\}$. These conditions can only
be met for $\vk=\vj'$, whence in fact at most one summand, namely for $\vj\in
I_{m,0}$ the one with $\vi=\vj$, and for $\vj'\in I_{m,0}$ the one with
$\vi=\vj'$, can be non-zero. So
\begin{equation*}
    Q(\vE_{\vj}+\vE_{\vk})+Q(\vE_{\vj})+Q(\vE_{\vk}) = 0 = [\vE_\vj,\vE_\vk]
    \mbox{~~~for~~~}\vk\neq\vj'
\end{equation*}
 and, irrespective of whether $\vi=\vj$ or $\vi=\vj'$, we have
\begin{equation*}
    Q(\vE_{\vj}+\vE_{\vj'})+Q(\vE_{\vj})+Q(\vE_{\vj'})
    =\frac{[\vE_\vj,\vE_{\vj}+\vE_{\vj'}][\vE_{\vj'},\vE_{\vj}+\vE_{\vj'}]}
                              {[\vE_\vj,\vE_{\vj'}]}
    =  [\vE_\vj,\vE_{\vj'}].
\end{equation*}
But this implies that the quadratic form $Q$ polarises to $\Alt$, \emph{i.~e.},
also the second property is satisfied.
\par
(c) Let $\widetilde Q$ be a quadratic form satisfying properties 1 and 2. Hence
the polar form of $Q-\widetilde Q=Q+\widetilde Q$ is zero. We consider $F$ as a
vector space over its subfield $F^\Box$ comprising all squares in $F$. So
$(Q+\widetilde Q):\TProdm \to F$ is a semilinear mapping with respect to the
field isomorphism $F\to F^\Box: x\mapsto x^2$; see, \emph{e.~g.},
\cite[p.~33]{dieu-71}. The kernel of $Q+\widetilde Q$ is a subspace of
$\TProdm$ which contains all decomposable tensors and, in particular, our basis
(\ref{eq:basis}). Hence $Q+\widetilde Q$ vanishes on $\TProdm$, and
$Q=\widetilde Q$ as required.
\end{proof}
\par
From (\ref{eq:def.Q}) and (\ref{eq:[].basis12}), the quadratic form $Q$ can
be written in terms of tensor coordinates $x_\vj\in F$ as
\begin{equation}\label{eq:koo.Q}
    Q\Big(\sum_{\vj\,\in\, I_m} x_{\vj}\vE_{\vj}\Big)
        = \sum_{\vi\,\in\, I_{m,0}}[\vE_\vi,\vE_{\vi'}] x_{\vi}x_{\vi'}
        = \prod_{k=1}^m [\ve_0\oben{k},\ve_1\oben{k}]
          \cdot \sum_{\vi\,\in\, I_{m,0}}x_{\vi}x_{\vi'}.
\end{equation}

The previous results may be slightly simplified by taking symplectic bases,
\emph{i.~e.}, $[\ve_0\oben{k},\ve_1\oben{k}]=1$ for all $k\in\{1,2,\ldots,m\}$,
whence also $[\vE_\vi,\vE_{\vi'}]=1$ for all $\vi\in I_m$.
\par
Observe also that Proposition~\ref{prop:1} fails to hold for $m=1$. A quadratic
form $Q$ vanishing for all decomposable tensors of $\vV_1$ is necessarily zero,
since any element of $\vV_1$ is decomposable. Hence the polar form of such a
$Q$ cannot be non-degenerate.

\begin{thm}\label{thm:1}
Let $m\geq 2$ and $\Char F=2$. There exists in the ambient space of the Segre
$\Segrem(F)$ a regular quadric $\cQ(F)$ with the following properties:
\renewcommand{\theenumi}{\emph{\arabic{enumi}}}
\begin{enumerate}
\item The projective index of $\cQ(F)$ is $2^{m-1}-1$.

\item $\cQ(F)$ is invariant under the group of projective collineations
    stabilising the Segre $\Segrem(F)$.
\end{enumerate}
\end{thm}
\begin{proof}
Any $f_k\in\GL(\vV_k)$, $k\in\{1,2,\ldots,m\}$, preserves the symplectic form
$\Alt$ on $\vV_k$ to within a non-zero factor. Any linear bijection $f_\sigma$
as in (\ref{eq:f-sigma}) is a symplectic transformation of $\TProdm$. Hence any
transformation from the group $\Stab{m}{F}$ preserves the symplectic form
(\ref{eq:[].basisfrei}) up to a non-zero factor. Consequently, also $Q$ is
invariant up to a non-zero factor under the action of $\Stab{m}{F}$.
\par
From (\ref{eq:koo.Q}) the linear span of the tensors $\vE_\vj$ with $\vj$
ranging in $I_{m,0}$ is a singular subspace with respect to $Q$. So the Witt
index of $Q$ equals $2^{m-1}$, because $\Alt$ being non-degenerate implies that
a greater value is impossible.
\par
We conclude that the quadric with equation $Q(\vX)=0$ has all the required
properties.
\end{proof}
We henceforth call $\cQ(F)$ the \emph{invariant quadric} of the Segre
$\Segrem(F)$. The case $m=2$ deserves special mention, as the Segre
$\cS_{1,1}(F)$ coincides with its invariant quadric $\cQ(F)$ given by
$Q(\sum_{\vj\in I_2}x_\vj\vE_\vj)=x_{00}x_{11}+x_{01}x_{10}=0$. This result
parallels the situation for $\Char F\neq 2$.

\section{The invariant basis}\label{se:basis}

In what follows $\bF_q$ will denote the Galois field with $q$ elements. We
adopt the notation and terminology from Section~\ref{se:background}, but we
restrict ourselves to the case $F=\bF_2$. Indeed, we shall always identify
$\bF_2$ with the prime field of $\bF_4=\{0,1,\omega,\omega^2\}$, where
$\omega^2+\omega+1=0$. For each $k\in\{1,2,\ldots,m\}$ we fix a basis
$(\ve_0\oben{k},\ve_1\oben{k})$ so that we obtain the tensor basis
(\ref{eq:basis}).
\par
Let $\vV$ be any vector space over $\bF_2$. Then $\vV$ can be embedded in
$\vW:={\bF_4}\otimes_{\bF_2}{\vV}$, which is a vector space over $\bF_4$, by
$\va\mapsto 1\otimes\va$; see, for example, \cite[p.~263]{kostrikin+m-89}. We
shall not distinguish between $\va$ and $1\otimes\va$. Likewise, if $f$ denotes
a linear mapping between vector spaces over $\bF_2$, then the unique
\emph{linear\/} extension of $f$ to the corresponding vector spaces over
$\bF_4$ will also be written as $f$ rather than $1\otimes f$. Similarly, we use
the same symbol for a bilinear form on $\vV$ and its extension to a bilinear
form on $\vW$. After similar identifications, we have
$\bP(\vV)\subset\bP(\vW)$, $\PGL(\vV)\subset\PGL(\vW)$, and so on. We make use
of the usual terminology for real and complex spaces also in our setting. We
address the vectors of $\vV$ to be \emph{real}, we speak of
\emph{complex{\trenn}conjugate} vectors, points, and subspaces. In particular,
a subspace is said to be \emph{real} if it coincides with its
complex{\trenn}conjugate subspace.
\par
Applying this extension to our vector spaces $\vV_k$ and their tensor product
$\TProdm$ gives vector spaces $\vW_k$ and $\bF_4\otimes_{\bF_2}
\big(\TProdm\big)$. The last vector space can be identified with $\TProdWm$ in
a natural way, so that the Segre $\Segrem(\bF_2)=:\Segrem(2)$ can be viewed as
as subset of $\Segrem(\bF_4)=:\Segrem(4)$. Likewise, we have
$\cQ(2):=\cQ(\bF_2)\subset\cQ(\bF_4)=:\cQ(4)$.
\par
From now on we shall make use of the following observation: When the projective
line $\bP(\vV_k)$ is embedded in $\bP(\vW_k)$ advantage can be taken from the
fact that there is a \emph{unique\/} projective basis consisting of the
complex{\trenn}conjugate pair of points, while there is a choice of three
different pairs of points for a projective basis of $\bP(\vV_k)$.

\begin{thm}\label{thm:2}
For each $k\in\{1,2,\ldots,m\}$ let $\bF_4\vu_0\oben{k}$ and
$\bF_4\vu_1\oben{k}$ be the only two points of the projective line
$\bP(\vW_k)=\PG(1,4)$ that are not contained in $\bP(\vV_{k})=\PG(1,2)$. Then
\begin{equation*}
    \cB_m:=\Big\{ \bF_4\vu_{i_1}\oben{1}\otimes\vu_{i_2}\oben{2}
    \otimes\cdots\otimes\vu_{i_m}\oben{m}
    \mid (i_1,i_2,\ldots,i_m)\in I_m \Big\}
\end{equation*}
is a basis of $\bP\big(\TProdWm\big)=\PG(2^{m}-1,4)$ which is invariant, as a
set, under the stabiliser $\Stab{m}{\bF_2}=:\Stab{m}{2}$ of the Segre
$\Segrem(2)$.
\end{thm}

\begin{proof}
We may assume that
\begin{equation}\label{eq:basis.w}
    \vu_0\oben{k} =\ve_0\oben{k}+\omega\ve_1\oben{k}\mbox{~~~~and~~~~}
    \vu_1\oben{k} =\ve_0\oben{k}+\omega^2\ve_1\oben{k}
    = (\ve_0\oben{k}+\ve_1\oben{k}) + \omega\ve_1\oben{k}.
\end{equation}
As $\vu_0\oben{k}$ and $\vu_1\oben{k}$ are linearly independent, the $2^m$
tensors
\begin{equation}\label{eq:tensorbasis.w}
    \vU_{i_1,i_2,\ldots,i_m}:=\vu_{i_1}\oben{1}\otimes\vu_{i_2}\oben{2}\otimes\cdots\otimes\vu_{i_m}\oben{m}
    \mbox{~~~with~~~} (i_1,i_2,\ldots,i_m)\in I_m
\end{equation}
constitute a basis of $\TProdWm$, whence $\cB_m$ is a projective basis. The
invariance of $\cB_m$ under $\Stab{m}{2}$ follows from the fact that the points
$\bF_4\vu_0\oben{k}$ and $\bF_4\vu_1\oben{k}$ are determined uniquely up to
relabelling.
\end{proof}

We shall refer to $\cB_m$ as the \emph{invariant basis} of the Segre
$\Segrem(2)$. In order to describe the action of the stabiliser group
$\Stab{m}{2}$ of the Segre $\Segrem(2)$ on the invariant basis we need a few
technical preparations:
\par
First, from now on the set $I_m=\{0,1\}^m$ of multi{\trenn}indices will be
identified with the vector space $\bF_2^m$. Secondly, for any $2$-dimensional
vector space $\vV$ over $\bF_2$ we can define the \emph{$\bF_2$-valued sign
function\/} $\sgn_2 :\GL(V)\to\bF_2$ to be $0$ if $f$ induces an even
permutation of $\vV\setminus\{0\}$ and $1$ otherwise.

\begin{prop}\label{prop:2}
The stabiliser group $\Stab{m}{2}$ of the Segre $\Segrem(2)$ has the following
properties:
\renewcommand{\theenumi}{\emph{\arabic{enumi}}}
\begin{enumerate}
\item Let $f_k\in\GL(\vV_k)$ for $k\in\{1,2,\ldots,m\}$ and write
    \begin{equation}\label{eq:schieb}
        \vs:=(\sgn_2 f_1,\sgn_2 f_2,\ldots,\sgn_2 f_m)\in\bF_2^m.
    \end{equation}
The collineation given by $f_1\otimes f_2 \otimes \cdots \otimes f_m$ sends
any point $\bF_4\vU_\vi\in\cB_m$ to the point $\bF_4\vU_{\vi+\vs}\in\cB_m$.

\item $\Stab{m}{2}$ acts transitively on the invariant basis $\cB_m$.

\item Let $\sigma\in S_{m}$ be a permutation and define $f_\sigma$ as in
    \emph{(\ref{eq:f-sigma})}. Then $f_\sigma$ sends any point\/
    $\bF_4\vU_\vi\in\cB_m$ to\/ $\bF_4\vU_{\sigma(\vi)}\in\cB_m$.
\end{enumerate}
\end{prop}

\begin{proof}
(a) Each mapping $f_k\in\GL(\vV_k)\subset\GL(\vW_k)$ induces a projectivity of
the projective line $\bP(\vW_k)=\PG(1,4)$ which stabilises
$\bP(\vV_k)=\PG(1,2)$.
\par
If $\sgn_2 f_k=0$ then the restriction to $\bP(\vV_k)$ is an even permutation,
namely either a permutation without fixed points or the identity on
$\bP(\vV_k)$. In the first case the characteristic polynomial of $f_k$ has two
distinct zeros over $\bF_4$, whence each of the two points $\bF_4\vu_0\oben{k}$
and $\bF_4\vu_1\oben{k}$ remains fixed. In the second case all points of
$\bP(\vV_k)$ are fixed.
\par
If $\sgn_2 f_k=1$ then $f_k$ gives a permutation of $\bP(\vV_k)$ with precisely
one fixed point. Such an $f_k$ is an involution, whence the points
$\bF_4\vu_0\oben{k}$ and $\bF_4\vu_1\oben{k}$ are interchanged.
\par
We infer from the above results that $f_1\otimes f_2 \otimes \cdots \otimes
f_m$ sends the point of $\cB_m$ with multi-index $\vi\in\bF_2^m$ to the point
of $\cB_m$ with multi-index $(\vi+\vs)\in\bF_2^m$.
\par
(b) Given $\vi,\vj\in \bF_2^m$ we let $\vs:=\vi+\vj$. In order find a
collineation from $\Stab{m}{2}$ taking $\bF_4\vU_\vi$ to $\bF_4\vU_\vj$, it
suffices to choose for all $k\in\{1,2,\ldots,m\}$ some $f_k\in\GL(\vV_k)$ with
$\sgn_2 f_k=s_k$. This can clearly be done, so that $f_1\otimes f_2 \otimes
\cdots \otimes f_m$ yields a collineation with the required properties.
\par
(c) According to the (basis{\trenn}dependent) definition of $f_\sigma$ in
(\ref{eq:f-sigma}), we have to consider the linear bijections $\vV_k\to
\vV_{\sigma(k)}$ sending $(\ve_0\oben{k},\ve_1\oben{k})$ to
$(\ve_0\oben{\sigma(k)},\ve_1\oben{\sigma(k)})$. By (\ref{eq:basis.w}), any
such map sends $(\vu_0\oben{k},\vu_1\oben{k})$ to
$(\vu_0\oben{\sigma(k)},\vu_1\oben{\sigma(k)})$. Now
$f_\sigma(\vU_\vi)=\vU_{\sigma(\vi)}$ follows immediately.
\end{proof}

The \emph{parity} of a point $\bF_4\vU_\vi\in\cB_m$ can be defined as the
parity of the multi-index $\vi$ (\emph{i.~e.}, it is even or odd according to
the number of $1$s among the entries of $\vi$). We write $\cB_m^+$ and
$\cB_m^-$ for the set of base points with even and odd parity, respectively.
Even though we can distinguish points of even and odd parity due to our fixed
bases $(\ve_0\oben{k},\ve_1\oben{k})$, a change of bases in the vector spaces
$\vV_k$ may alter the parity of a point. But \emph{having the same parity\/} is
an equivalence relation on $\cB_m$ with two equivalence classes, namely
$\cB_m^+$ and $\cB_m^-$, each of cardinality $2^{m-1}$.
\par

We define the Hamming distance between $\bF_4\vU_\vi$ and $\bF_4\vU_\vj$ as the
Hamming distance of their multi{\trenn}indices $\vi$ and $\vj$. In particular,
we speak of \emph{opposite} points if $\vi$ and $\vj$ are opposite. For each
point of $\cB_m$ there is a unique opposite point. By (\ref{eq:basis.w}) and
(\ref{eq:tensorbasis.w}) opposite points of $\cB_m$ are
complex{\trenn}conjugate with respect to the Baer subspace
$\bP\big(\TProdm\big)$ of $\bP\big(\TProdWm\big)$. The opposite point to
$\bF_4\vU_\vi$ can also be characterised as the only point
$\bF_4\vU_\vj\in\cB_m$ such that $[\vU_\vi,\vU_\vj]\neq 0$. We remark that the
Hamming distance on $\cB_m$ admits another description due to
$\cB_m\subset\Segrem(4)$. The Hamming distance of $p,q\in\cB_m$ equals the
number of lines on a shortest polygonal path in $\Segrem(4)$ from $p$ to $q$.
\par
From the proof of Proposition~\ref{prop:2} and the above remarks we immediately
obtain:

\begin{thm}\label{thm:3}
The stabiliser group $\Stab{m}{2}$ of the Segre $\Segrem(2)$ acts on the
invariant basis $\cB_m$ (via the multi{\trenn}indices of its points) as the
group of all affine transformations of\/ $\bF_2^m$ having the form $\vi\mapsto
\sigma(\vi)+\vs$, where $\sigma\in S_m$ and $\vs\in\bF_2^m$. Hence, Hamming
distances on $\cB_m$ are preserved under $\Stab{m}{2}$, and the partition
$\cB_m=\cB_m^+\mathrel{\dot\cup}\cB_m^-$ is a $\Stab{m}{2}${\trenn}invariant
notion.
\end{thm}

We now use the invariant basis for describing some other
$\Stab{m}{2}${\trenn}invariant subsets. In the following theorem we also make
use of a particular property of Segre varieties over $\bF_2$. Recall that (for
an arbitrary ground field $F$) there are precisely $m$ generators through any
point $p$ of the Segre $\Segrem(F)$. They span the ($m$-dimensional)
\emph{tangent space} of $\Segrem(F)$ at $p$. The \emph{tangent} lines at $p$
are the lines through $p$ which lie in its tangent space. For $F=\bF_2$ there
are $2^m-1$ tangents at $p$. Precisely one of them does not lie in any of the
$(m-1)$-dimensional subspaces which are spanned by $m-1$ generators through
$p$. This line will be called the \emph{distinguished tangent} at $p$.

\begin{thm}\label{thm:4}
The stabiliser group $\Stab{m}{2}$ of the Segre $\Segrem(2)$ has the following
properties:
\renewcommand{\theenumi}{\emph{\arabic{enumi}}}
\begin{enumerate}

\item The union of the skew subspaces $\spn\cB_m^+$ and $\spn\cB_m^-$ is a
    $\Stab{m}{2}${\trenn}invariant subset of\/ $\bP\big(\TProdWm\big)$.

\item The union of the $2^{m-1}$ mutually skew real lines\footnote{Any
    line joining complex{\trenn}conjugate points is real (cf.\ the beginning of
    Section~\ref{se:basis}). It carries three real points.
    We use the symbol $\vee$ for the join of points.}
    \begin{equation}\label{eq:U+U'}
        \bF_4\vU_\vi \vee \bF_4\vU_{\vi'} \mbox{~~~with~~~}\vi\in I_{m,0}
    \end{equation}
    is a $\Stab{m}{2}${\trenn}invariant subset. The $3\cdot 2^{m-1}$ real
    points on these lines comprise an orbit of $\Stab{m}{2}$.

\item If $m$ is even then $\spn\cB_m^+$ and $\spn\cB_m^-$ are real
    subspaces. Each of the lines from \emph{(\ref{eq:U+U'})} is contained
    in precisely one of them.

\item If $m$ is odd then $\spn\cB_m^+$ and $\spn\cB_m^-$ are
    complex{\trenn}conjugate subspaces. All lines from
    \emph{(\ref{eq:U+U'})} meet $\spn\cB_m^+$ and $\spn\cB_m^-$ at
    precisely one point, respectively.

\item All distinguished tangents of the Segre $\Segrem(2)$ meet
    $\spn\cB_m^+$ and $\spn\cB_m^-$ at precisely one point, respectively.

\end{enumerate}
\end{thm}

\begin{proof}
Ad 1 and 2: The assertions on the invariance of
$\spn\cB_m^+\mathrel{\dot\cup}\spn\cB_m^-$ and on the invariance of the union
of all lines from (\ref{eq:U+U'}) are a direct consequence of
Theorem~\ref{thm:3}.
\par
We denote the set of all real points on the lines from (\ref{eq:U+U'}) by
$\cR$. Let $\vj\in I_{m,0}$ and let $p$ be an arbitrary real point on the line
$\bF_4\vU_\vj \vee \bF_4\vU_{\vj'}$. Any collineation from $\Stab{m}{2}$ takes
$p$ to some real point on a line from (\ref{eq:U+U'}), whence the orbit of $p$
is contained in $\cR$.
\par
Conversely, let $q\in\cR$. So there is a $\vk\in I_{m,0}$ with $q
\in\bF_4\vU_\vk \vee \bF_4\vU_{\vk'}$. By Theorem~\ref{thm:3}, there exists a
collineation in $\Stab{m}{2}$ which maps $\bF_4\vU_\vj$ to $\bF_4\vU_\vk$ and,
consequently, also $\bF_4\vU_{\vj'}$ to $\bF_4\vU_{\vk'}$. Furthermore, $p$ is
mapped to some real point $\tilde p$ on the line $\bF_4\vU_\vk \vee
\bF_4\vU_{\vk'}$. There exists $f_1\in\GL(\vV_1)$ with $\sgn_2 f_1 =0$. Then
$\vu_0\oben{1}$ and $\vu_1\oben{1}$ are eigenvectors of $f_1$ with eigenvalues
$\lambda$ and $\lambda^2$, where $\lambda\in\{\omega,\omega^2\}$. (See the
proof of Proposition~\ref{prop:2}.) The linear bijection $f:=
f_1\otimes\id_{\vV_2}\otimes\cdots\otimes\id_{\vV_m}$ has $\vU_\vk$ and
$\vU_{\vk'}$ as eigenvectors with eigenvalues $\lambda$ and $\lambda^2$,
respectively, due to $k_1=0$. Thus the collineation arising from $f$ induces a
non-identical even permutation on the three real points of the line
$\bF_4\vU_\vk \vee \bF_4\vU_{\vk'}$. Such a permutation has only one cycle. So
one of $f$, $f^2$ or $f^3$ yields a collineation from $\Stab{m}{2}$ which maps
$\tilde p$ to $q$.

\par
Ad 3 and 4: Opposite points of $\cB_m$ are complex{\trenn}conjugate and vice
versa. Such points share the same parity for $m$ even, but have different
parity for $m$ odd.

\par
Ad 5: First, we exhibit the distinguished tangent $T$ of the Segre
$\Segrem{(2)}$ at the point $\bF_2\vE_{1,1,\ldots,1}$. On each of the $m$
generators of the Segre through this point we select one more real point,
namely $\bF_2\vE_{0,1,1,\ldots,1}$, $\bF_2\vE_{1,0,1,\ldots,1}$, \ldots,
$\bF_2\vE_{1,1,\ldots,1,0}$ for facilitating our subsequent reasoning. So, the
distinguished tangent $T$ contains the real point
\begin{equation}\label{eq:tangentenpkt1}
    \bF_2(\vE_{0,1,1,\ldots,1}+\vE_{1,0,1,\ldots,1}+\cdots+\vE_{1,1,\ldots,1,0}).
\end{equation}
By (\ref{eq:basis.w}), we have $\ve_0\oben{k}=
\omega^2\vu_0\oben{k}+\omega\vu_1\oben{k}$ and $\ve_1\oben{k}=
\vu_0\oben{k}+\vu_1\oben{k}$ for all $k\in\{1,2,\ldots,m\}$. So
$\vE_{1,1,\ldots,1}=\sum_{\vi\,\in\, I_m}\vU_\vi$ and
\begin{equation*}
    \vE_{0,1,\ldots,1}=\sum_{\vi\,\in\, I_m} x_\vi\oben{1}\vU_\vi
    \mbox{~~with~~}x_\vi\oben{1}=
    \left\{
    \begin{array}{lcr}
        \omega^2 &\mbox{for}& i_1=0,\\
        \omega   &\mbox{for}& i_1=1.
    \end{array}
    \right.
\end{equation*}
\emph{Mutatis mutandis}, we obtain linear combinations for
$\vE_{1,0,1,\ldots,1},\ldots,\vE_{1,1,\ldots,1,0}$ with coefficients
$x_\vi\oben{2},\ldots,x_\vi\oben{m}\in\{\omega^2,\omega\}$. Summing up gives
\begin{equation}\label{eq:tangentenpunkt2}
    \vE_{0,1,1,\ldots,1}+\vE_{1,0,1,\ldots,1}+\cdots+\vE_{1,1,\ldots,1,0}=
    \sum_{\vi\,\in\, I_m}y_\vi\vU_\vi,
\end{equation}
where
\begin{equation*}
    y_\vi=\underbrace{\omega^2+\omega^2+\cdots+\omega^2}
            _{\mbox{\footnotesize\# of zeros in $\vi$}}
    +\underbrace{\omega+\omega+\cdots+\omega}
            _{\mbox{\footnotesize\# of ones in $\vi$}} .
\end{equation*}
There are two cases:
\par
\emph{$m$ even:} Due to $\Char\bF_4=2$ we have $y_\vi=0$ for all $\vi$ with
even parity and $y_\vi=\omega^2+\omega=1$ for all $\vi$ with odd parity. So $T$
meets $\spn{\cB_m^-}$ at the point (\ref{eq:tangentenpkt1}). The sum of the
tensor from (\ref{eq:tangentenpunkt2}) and $\vE_{1,1,\ldots,1}$ determines the
point of intersection of $T$ with $\spn\cB_m^+$.
\par
\emph{$m$ odd:} Due to $\Char\bF_4=2$ we have $y_\vi=\omega^2$ for all $\vi$
with even parity and $y_\vi=\omega$ for all $\vi$ with odd parity. So $T$ meets
$\spn{\cB_m^+}$ at the point
\begin{equation}\label{eq:E_spur}
    \bF_4\Big(\sum_\vj \omega^2\vU_\vj\Big)=\bF_4\Big(\sum_\vj \vU_\vj\Big),
\end{equation}
where $\vj$ ranges over all elements of $I_m$ with even parity, and the
subspace $\spn{\cB_m^-}$ at the point $\bF_4(\sum_\vk \omega
\vU_\vk)=\bF_4(\sum_\vk \vU_\vk)$, where $\vk$ ranges over all elements of
$I_m$ with odd parity.
\par
In either case the two points of intersection are unique, because $\spn\cB_m^+$
and $\spn\cB_m^-$ are skew.
\par
Next, we consider an arbitrary distinguished tangent of the Segre. As all
points of the Segre comprise a point orbit of $\Stab{m}{2}$, also all
distinguished tangents are in one line orbit of $\Stab{m}{2}$. So, all
distinguished tangents share the properties of the tangent $T$.
\end{proof}

Any pair of skew and complex{\trenn}conjugate subspaces of
$\bP\big(\TProdWm\big)$ determines a \emph{geometric line spread} of
$\bP\big(\TProdm\big)$. This spread comprises all real lines which meet one of
the subspaces (and hence both of them in complex{\trenn}conjugate points). Any
of these subspaces is called an \emph{indicator set} of the spread. See
\cite[p.~74]{baer-63a} and \cite[p.~29]{segre-64a}. So, part of
Theorem~\ref{thm:4} can be reformulated as follows:

\begin{cor}\label{cor:1}
If $m$ is odd then the complex{\trenn}conjugate subspaces $\spn\cB_m^+$ and
$\spn\cB_m^-$ are indicator sets of a $\Stab{m}{2}${\trenn}invariant geometric
line spread $\cL$ of\/ $\bP\big(\TProdm\big)=\PG(2^{m}-1,2)$. All distinguished
tangents of the Segre $\Segrem{(2)}$ and all lines given by
\emph{(\ref{eq:U+U'})} belong to this spread.
\end{cor}

It is now a straightforward task to establish connections between the
fundamental polarity of the Segre $\Segrem{(4)}$ and the quadric $\cQ(4)$ which
arises according to Theorem~\ref{thm:1}. The reader will easily verify the
following: The fundamental polarity maps each of the lines from (\ref{eq:U+U'})
to the span of the remaining ones. For any even $m$ the subspaces $\spn\cB_m^+$
and $\spn\cB_m^-$ are interchanged under the fundamental polarity, whereas for
$m$ odd each of them is invariant (totally isotropic). The subspaces
$\spn\cB_m^+$ and $\spn\cB_m^-$ are contained in $\cQ(4)$ precisely when $m\geq
2$ is odd.

\begin{exa}\label{exa:m=2}
Let $m=2$. The Segre $\cS_{1,1}(2)$ is a hyperbolic quadric of
$\bP(\vV_1\otimes\vV_2)=\PG(3,2)$. The $15$ points of this projective space
fall into two orbits under the action of $G_{1,1}(2)$: The first orbit is
$\cS_{1,1}(2)$ (nine points), the remaining six points form the second orbit.
It comprises the real points of the lines $\spn\cB_m^+$ and $\spn\cB_m^-$.
There are nine distinguished tangents of the quadric. Together they form the
hyperbolic linear congruence of lines which arises by joining every real point
of $\spn\cB_m^+$ with every real point of $\spn\cB_m^-$.
\end{exa}
The previous example is clearly just an easy exercise, and it could be mastered
without our results about the general case.

\section{The invariant Hermitian variety}\label{se:hermite}

\par
The symplectic form $\Alt$ on $\TProdm$ can be extended in exactly two ways to
a non-degenerate sesquilinear form\footnote{We assume such forms to be linear
in the right and semilinear in the left argument.} on $\TProdWm$. The bilinear
extension is symplectic. In accordance with the notation used elsewhere, it is
also denoted by $\Alt$. The only other extension is sesquilinear with respect
to the Frobenius automorphism $z\mapsto z^2$ of $\bF_4$. Such Hermitian
extension will be written as $\Alt_H$. We have
\begin{equation}\label{eq:[]_H}
    [\vX,\vY]_H=[\overline\vX,\vY]
\end{equation}
for all tensors $\vX,\vY\in\TProdWm$, where $\overline\vX$ denote the
complex{\trenn}conjugate tensor of $\vX$. While $\Alt$ describes the
fundamental polarity of the Segre $\Segrem(4)$, the Hermitian sesquilinear form
$\Alt_H$ yields a unitary polarity of $\bP(\TProdWm)$ and, moreover, the
Hermitian variety $\cH$ comprising all its absolute points. By its definition,
$\cH$ is a $\Stab{m}{2}${\trenn}invariant notion, whence we call it the
\emph{invariant Hermitian variety} of the Segre $\Segrem{(2)}$. Note that
$\cH$, like the invariant basis and the invariant line spread, is an invariant
notion only for $\Segrem{(2)}$, but not for $\Segrem{(4)}$.
\par
We remark that the invariant basis $\cB_m$ is self-polar with respect to the
unitary polarity given by $\Alt_H$. Indeed, given $\vi,\vj\in\bF_2^m$ we have
\begin{eqnarray}\label{eq:U-Hgleich}
    [\vU_\vi,\vU_\vi]_H & = &
    \prod_{k=1}^m [\vu_{i_k}\oben{k},\overline\vu_{i_k}\oben{k}] =
    (\omega+\omega^2)\prod_{k=1}^m [\ve_{0}\oben{k},\overline\ve_{1}\oben{k}]=1,\\
    \label{eq:U-Hungleich}
    [\vU_\vi,\vU_\vj]_H &=&
    \prod_{k=1}^m [\vu_{i_k}\oben{k},\overline\vu_{j_k}\oben{k}]=0
    \mbox{~~~~for all~~~~}\vi\neq\vj,
\end{eqnarray}
since, for example, $i_1\neq j_1$ implies $\vu_{i_1}\oben{1} =
\overline\vu_{j_1}\oben{1}$, whence $[\vu_{i_1}\oben{1},
\overline\vu_{j_1}\oben{1}]=0$.

The following Proposition establishes a link among the invariant line spread
from Corollary~\ref{cor:1}, the invariant quadric $\cQ(2)$, and the invariant
Hermitian variety $\cH$.

\begin{prop}\label{prop:3}
Let $m\geq 2$ be odd. A line $L$ of the invariant geometric line spread $\cL$
is a generator of the invariant quadric $\cQ(4)$ if, and only if, the
intersection point $L\cap\spn\cB_m^+$ belongs to the invariant Hermitian
variety $\cH$. Otherwise that line $L$ is a bisecant of $\cQ(4)$, whence it has
no points in common with $\cQ(2)$.
\end{prop}

\begin{proof}
Suppose that $L\cap\spn\cB_m^+=\bF_4\vX$. By a remark at the end of
Section~\ref{se:basis}, we have $\bF_4\vX\in\spn\cB_m^+\subset\cQ(4)$ and
$\bF_4\overline\vX\in\spn\cB_m^-\subset\cQ(4)$. So the line $L=\bF_4\vX\vee
\bF_4\overline\vX$ is either a generator or a bisecant of $\cQ(4)$. The first
possibility occurs precisely when $\bF_4\overline\vX$ lies in the tangent
hyperplane of $\cQ(4)$ at $\bF_4\vX$. Employing (\ref{eq:[]_H}), this in turn
is equivalent to $0=[\overline\vX,\vX]=[\vX,\vX]_H$ which characterises
$\bF_4\vX$ as a point of $\cH$.
\end{proof}
The above Proposition is a special case of \cite[Theorem~1]{dye-86a} on
quadrics which admit a spread of lines. In our context the invariant spread
from Corollary~\ref{cor:1} yields a spread of lines on $\cQ(2)$, since over
$\bF_2$ each line of the invariant spread is either external to or contained in
that quadric.

\section{The Segre variety $\Segre(2)$}\label{se:m=3}

In this section we exhibit the ambient space
$\bP(\vV_1\otimes\vV_2\otimes\vV_3)=\PG(7,2)$ of the Segre $\Segre(2)$. This
space has $2^8-1=255$ points. Furthermore, we have the cardinalities
$\#\Segre(2)=3^3=27$, $\#\cQ(2)=(2^3+1)(2^4-1)=135$ (see
\cite[Theorem~5.21]{hirschfeld-98}), and $\#\cL=255/3=85$.

\begin{prop}\label{prop:4}
Under the action of the stabiliser group $G_{\Segre(2)}$ of the Segre
$\Segre(2)$ the lines of the invariant spread $\cL$ of\/
$\bP(\vV_1\otimes\vV_2\otimes\vV_3)=\PG(7,2)$ fall into four orbits
$\cL_1,\cL_2,\cL_3,\cL_4$. In terms of the invariant basis $\cB_3$ the
following characterisation holds: A line from $\cL$ is in orbit $\cL_r$ if, and
only if, its (imaginary) point of intersection with the subspace $\spn\cB_3^+$
lies in $4-r$ planes of the tetrahedron $\cB_3^+$.
\end{prop}

\begin{proof}
(a) Throughout this proof the pointwise stabiliser and the stabiliser of
$\cB_3^+$ in the group $G_{\Segre(2)}$ are abbreviated by $G_{\mathrm{pw}}^+$
and $G^+$, respectively. We observe that $G_{\mathrm{pw}}^+$ acts transitively
on the points of the Segre $\Segre(2)$: We fix the point
$\bF_2\vE_{111}=\bF_2(\ve_1\oben{1}\otimes\ve_1\oben{2}\otimes\ve_1\oben{3})$.
Given any point of the Segre, say $\bF_2\vA$, where
$\vA=\va_1\otimes\va_2\otimes\va_3$, there are linear bijections
$f_k\in\GL(\vV_k)$ satisfying $\sgn_2 f_k=0$ and $\ve_1\oben{k}\mapsto\va_k$
for $k=\{1,2,3\}$. So $f_1\otimes f_2\otimes f_3$ induces a collineation which
sends $\bF_2\vE_{111}$ to $\bF_2\vA$ and belongs to $G_{\mathrm{pw}}^+$ by
(\ref{eq:schieb}).
\par
We write $\cM_r$, $r\in\{1,2,3,4\}$, for the subset of $\spn\cB_3^+=\PG(3,4)$
comprising all points which lie in precisely $4-r$ planes of the tetrahedron
$\cB_3^+$. So we have $\#\cM_1=4$ vertices, $\#\cM_2=3\cdot 6=18$ edge points,
$\#\cM_3=4\cdot 9=36$ face points, and $\#\cM_4=27$ general points. Clearly,
the $G^+$-orbit of any point from $\spn\cB_3^+$ is contained in one of the sets
$\cM_r$.

(b) We show that $\cM_4$ is an orbit under $G_{\mathrm{pw}}^+$: By
(\ref{eq:E_spur}), the distinguished tangent of the Segre at $\bF_2\vE_{111}$
meets $\spn\cB_3^+$ at the point
$p:=\bF_4(\vU_{000}+\vU_{011}+\vU_{101}+\vU_{110})\in\cM_4$. We infer from the
transitive action of $G_{\mathrm{pw}}^+$ on the Segre $\Segre(2)$ that all
distinguished tangents meet $\spn\cB_3^+$ in points of $\cM_4$. Since
$\#\Segre(2)=27=\#\cM_4$, the group $G_{\mathrm{pw}}^+$ acts transitively on
$\cM_4$.

(c) Any edge of $\cB_3^+$ contains precisely three points of $\cM_2$. We obtain
all of them by projecting $\cM_4$ from the opposite edge, whence
$G_{\mathrm{pw}}^+$ acts transitively on the set of these three points.
Likewise, $G_{\mathrm{pw}}^+$ acts transitively on the nine points of $\cM_3$
in any face of $\cB_3^+$.

(d) We know from Proposition~\ref{prop:2} that $G^+$ acts transitively on the
set of vertices of $\cB_3^+$ via translations $\vi\mapsto \vi+\vs$ on
multi{\trenn}indices. From Theorem~\ref{thm:3} the $G^+$-stabiliser of
$\bF_4\vU_{000}$ acts transitively on the remaining vertices of $\cB_3^+$ via
permutations $\vi\mapsto \sigma(\vi)$ on multi{\trenn}indices. Together with
our previous results this means that each of the four subsets $\cM_r$ is a
$G^+$-orbit. Consequently, each of the sets $\cL_r$ is contained in an orbit
under the action of $G_{\Segre(2)}$ on the line spread $\cL$.
\par
Any collineation from $G_{\Segre(2)}\setminus G^+$ also preserves each of the
sets $\cL_r$, as it commutes with the Baer involution of $\PG(7,4)$ fixing
$\bP(\vV_1\otimes\vV_2\otimes\vV_3)=\PG(7,2)$ pointwise. This completes the
proof.
\end{proof}

From (\ref{eq:U-Hgleich}) and (\ref{eq:U-Hungleich}) the equation of the
Hermitian variety $\cH\cap\spn\cB_3^+$ with respect to the basis
$(\vU_{000},\vU_{011},\vU_{101},\vU_{110})$ reads
\begin{equation*}
    x_{000}^3+x_{011}^3 + x_{101}^3 + x_{110}^3=0.
\end{equation*}
Because of $z^3=1$ for all $z\in\bF_4\setminus\{0\}$, we get
$\cH\cap\spn\cB_3^+ = \cM_2\cup\cM_4$. By Proposition~\ref{prop:3}, the lines
from $\cL_2\cup\cL_4$ are on the invariant quadric $\cQ(4)$. More precisely,
the lines from $\cL_2$ are those generators of $\cQ(4)$ which do not contain
any point of the Segre $\Segre(2)$, whereas the lines from $\cL_4$ are the
distinguished tangents of $\Segre(2)$.
\begin{figure}[ht!]\unitlength0.7cm
\centering
\begin{picture}(17.8,7.8) 
     \put(0,0){\includegraphics[height=7.8\unitlength]{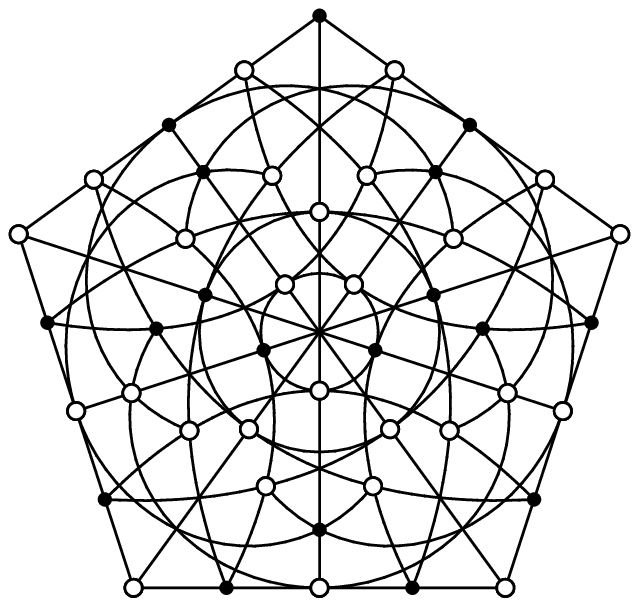}}
     \put(10,0){\includegraphics[height=7.8\unitlength]{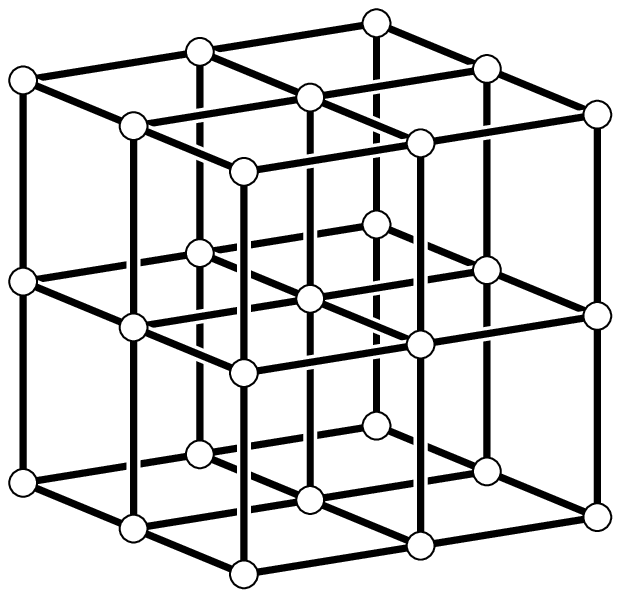}}
\end{picture}
\caption{The Hermitian variety $\cM_2\cup\cM_4$ (left) and the Segre
$\Segre(2)$ (right).}\label{fig:1}
\end{figure}
Figure~\ref{fig:1}, left\footnote{The style of our figure is taken from
\cite[p.~61]{polster-98a}.}, displays the polar space (point-line incidence
structure) on $\cH$. This Hermitian variety consists of $45$ points and carries
$27$ lines (represented by segments and curves), with five points on any line
and three lines on any point. The $27$ points represented by small circles are
those from $\cM_4$, the remaining $18$ points are represented by bullets and
belong to $\cM_2$. The $27$ points of $\cM_4$ can be viewed as a \emph{skew
projection\/} of the Segre $\Segre(2)$ (Figure~\ref{fig:1}, right) into $\cH$
along the invariant line spread. Under this projection collinearity of points
is being preserved.
\par
The lines from $\cL_1$ are the four lines from (\ref{eq:U+U'}). Like the
remaining $36$ lines from $\cL_3$ they are exterior lines (over $\bF_2$) of the
invariant quadric $\cQ(2)$.

\begin{prop}\label{prop:5}
Under the action of the stabiliser group $G_{\Segre(2)}$ of the Segre
$\Segre(2)$ the points of\/ $\bP(\vV_1\otimes\vV_2\otimes\vV_3)=\PG(7,2)$ fall
into five orbits $\cO_1,\cO_2,\ldots,\cO_5$. For $r\in\{1,2,3\}$ the points of
$\cO_r$ are precisely the real points on the lines of $\cL_r$. The orbit
$\cO_4$ comprises those real points on the lines from $\cL_4$ which are off the
Segre $\Segre(2)$, whereas $\cO_5$ equals the Segre $\Segre(2)$.
\end{prop}

\begin{proof}
It is clear that $\cO_5$ is an orbit under $G_{\Segre(2)}$. The points of
$\cO_1$ form an orbit according to Theorem~\ref{thm:4}. In order to show that
$\cO_2$ and $\cO_3$ are orbits, we shall select one line of $\cL_2$ and
$\cL_3$, respectively. By Proposition~\ref{prop:4}, it suffices then to show
that all real points of this line are in one orbit. This task will be
accomplished with mappings $f_k\in\GL(V_k)$, $k\in\{1,2,3\}$, given by
$\ve_0\oben{k}\mapsto\ve_1\oben{k}$, $\ve_1\oben{k}\mapsto
\ve_0\oben{k}+\ve_1\oben{k}$. From (\ref{eq:basis.w}), we have
$f_k(\vu_0\oben{k})=\omega\vu_0\oben{k}$ and
$f_k(\vu_1\oben{k})=\omega^2\vu_1\oben{k}$.
\par
Let $L_2\in\cL_2$ be the line joining $\bF_4(\vU_{000}+\vU_{011})\in\cM_2$ with
its complex{\trenn}conjugate point $\bF_4(\vU_{111}+\vU_{100})$. The mapping
$f_1\otimes\id_{\vV_2}\otimes\id_{\vV_3}$ has $\vU_{000}+\vU_{011}$ as
eigentensor with eigenvalue $\omega$. Its complex{\trenn}conjugate tensor is
therefore an eigentensor with eigenvalue $\omega^2$. From the proof of
Proposition~\ref{prop:2}, this implies that
$f_1\otimes\id_{\vV_2}\otimes\id_{\vV_3}$ induces a non-trivial even
permutation on the set of real points of $L_2$. So, under the powers of this
permutation the three real points of $L_2$ are permuted in one cycle.
\par
Let $L_3\in\cL_3$ be the line joining
$\bF_4(\vU_{011}+\vU_{101}+\vU_{110})\in\cM_3$ with its
complex{\trenn}conjugate point. Here $f_1\otimes f_2\otimes f_3$ possesses
$\vU_{011}+\vU_{101}+\vU_{110}$ as eigentensor with eigenvalue
$\omega^5=\omega^2$. Its complex{\trenn}conjugate tensor is therefore an
eigentensor with eigenvalue $\omega$. Now the assertion follows as above.

\par
The distinguished tangent of the Segre $\Segre(2)$ at the point
$\bF_2\vE_{111}$ contains two precisely two points of $\cO_4$. From
(\ref{eq:tangentenpkt1}), these points are
$\bF_2(\vE_{011}+\vE_{101}+\vE_{110})$ and
$\bF_2(\vE_{111}+\vE_{011}+\vE_{101}+\vE_{110})$. Let $g_1\in\GL(\vV_1)$ be
defined by $\ve_0\oben{1}\mapsto\ve_0\oben{1}+\ve_1\oben{1}$,
$\ve_1\oben{1}\mapsto \ve_1\oben{1}$. Then
$g_1\otimes\id_{\vV_2}\otimes\id_{\vV_3}$ will interchange these two points,
whence we may argue as before.
\end{proof}

Let us close this section with a few remarks: The orbits of the stabiliser
group $G_{\Segre(2)}$ are described (without proof) in a completely different
way in \cite[p.~82]{glynn+g+m+g-06z}. The $a,b,c,d,e$-orbits from there are in
our terminology the sets $\cO_5$ ($27$ points), $\cO_2$ ($54$ points), $\cO_3$
($108$ points), $\cO_4$ ($54$ points), and $\cO_1$ ($12$ points), respectively.
The union $\cO_2\cup\cO_4\cup\cO_5$ is the invariant quadric $\cQ(2)$. With
respect to the tensor basis (\ref{eq:basis}) the equation of $\cQ(2)$ reads
\begin{equation}\label{eq:Q3}
    x_{000}x_{111}+x_{001}x_{110}+x_{010}x_{101}+x_{011}x_{100}=0.
\end{equation}
The square of the left hand side of (\ref{eq:Q3}) is \emph{Cayley's
hyperdeterminant} of the $3\times 3\times 3$ array $ (x_\vi)_{\vi\,\in\, I_3}$;
see \cite[Theorem~5.45]{glynn+g+m+g-06z} and compare with \cite{gelfand+k+z-94}
and \cite{glynn-98}.
\par
By virtue of the fundamental polarity of $\Segre(2)$, Proposition~\ref{prop:5}
provides a classification of the hyperplanes of
$\bP(\vV_1\otimes\vV_2\otimes\vV_3)$ under the action of the group
$G_{\Segre(2)}$. Moreover, it gives a classification of the \emph{geometric
hyperplanes} (or \emph{primes}) of $\Segre(2)$, since any geometric hyperplane
of this Segre arises as intersection with a unique hyperplane of the ambient
space \cite{ronan-87}. This is a rather particular property of Segre varieties
$\Segrem(2)$ which is not shared by Segre varieties $\Segrem(F)$ in general
\cite{bichara+m+z-97a}.

\par
The Segre $\Segre(2)$ (as a point-line geometry) appears in the literature in
various guises, namely as the $(27_3,27_3)$ \emph{Gray configuration}
\cite{maru+p+w-05a} or as the \emph{smallest slim dense near hexagon}
\cite{brou+c+h+w-94a}. It is also a point model of the \emph{chain geometry}
based on the $\bF_2$-algebra $\bF_2\times\bF_2\times\bF_2$, the chains being
the twisted cubics on $\Segre(2)$ (\emph{i.~e.} triads of points with mutual
Hamming distance $3$); see \cite[(5.4)]{benz+s+s-81} or
\cite[p.~272]{blunck+he-05}. We add in passing that the tangent lines of these
twisted cubics are just our distinguished tangents of $\Segre(2)$.

\section{Conclusion}

We established several invariant notions for Segre varieties $\Segrem(2)$ over
the field $\bF_2$. For $m\leq 3$ these invariants provide sufficient
information for the classification of the points and hyperplanes of the ambient
space of $\Segrem(2)$. For larger values of $m$ the situation seems to be much
more intricate. For example, when $m$ is odd then the lines of the invariant
spread will fall into at least $2^{m-1}$ orbits, as follows from a
straightforward generalisation of Proposition~\ref{prop:4}. However, this gives
only a lower bound for the number of orbits. Indeed, for $m\geq 3$ there are
$3^m$ distinguished tangents of $\Segrem(2)$, but $3^{2^{m-1}-1}$ points of
$\spn\cB_m^+$ which belong to no face of the simplex $\cB_m^+$. These two
cardinalities coincide only when $m=3$, whence for all odd $m>3$ we no longer
have a one-one correspondence between the set of distinguished tangents and the
set of all points of $\spn\cB_m^+$ which belong to no face of $\cB_m^+$.

\paragraph{Acknowledgements.} This work emerged out of vivid discussions
amongst the members of the Cooperation Group ``Finite Projective Ring
Geometries: An Intriguing Emerging Link Between Quantum Information Theory,
Black-Hole Physics, and Chemistry of Coupling'' at the Center for
Interdisciplinary Research (ZiF), University of Bielefeld, Germany. It was
carried out within the ``Slovak-Austrian Science and Technology Cooperation
Agreement'' under grants SK~07-2009 (Austrian side) and SK-AT-0001-08 (Slovak
side), being also partially supported by the VEGA grant agency projects
2/0092/09 and 2/0098/10.


\small
\newcommand{\Dbar}{\makebox[0cm][c]{\hspace{-2.5ex}\raisebox{0.25ex}{-}}}

\noindent
Hans Havlicek\\
Institut f\"{u}r Diskrete Mathematik und Geometrie\\
Technische Universit\"{a}t\\
Wiedner Hauptstra{\ss}e 8--10/104\\
A-1040 Wien\\
Austria\\
\texttt{havlicek@geometrie.tuwien.ac.at}
\\~\\
\noindent
Boris Odehnal\\
Institut f\"{u}r Diskrete Mathematik und Geometrie\\
Technische Universit\"{a}t\\
Wiedner Hauptstra{\ss}e 8--10/104\\
A-1040 Wien\\
Austria\\
\texttt{boris@geometrie.tuwien.ac.at}
\\~\\
\noindent
Metod Saniga\\
Astronomical Institute\\
Slovak Academy of Sciences\\
SK-05960 Tatransk\'{a} Lomnica\\
Slovak Republic\\
\texttt{msaniga@astro.sk}

\end{document}